\documentclass[10pt]{amsart}
\usepackage{amsmath,amssymb,mathtools}
\usepackage[T1]{fontenc}
\usepackage{lmodern,textcomp}
\usepackage{color}
\usepackage{graphicx}
\usepackage{amsthm}
\usepackage[all]{xy}
\usepackage{caption}
\usepackage{enumitem}
\usepackage{graphicx}

\usepackage[pdfborder={0 0 0}]{hyperref}
\usepackage[capitalize]{cleveref}

\def\H{{\mathcal{H}}}
\def\M{{\mathcal{M}}}

\def\oM{\overline{\mathcal{M}}}

\def\oH{\overline{\mathcal{H}}}

\def\CC{\mathbb{C}}
\def\ZZ{\mathbb{Z}}
\def\NN{\mathbb{N}}
\def\QQ{\mathbb{Q}}
\def\PP{\mathbb{P}}

\theoremstyle{definition} 
\newtheorem{definition}{Definition}[section]

\newtheorem{remark}[definition]{Remark}

\theoremstyle{plain}

\newtheorem{theorem}[definition]{Theorem}

\newtheorem{proposition}[definition]{Proposition}

\newtheorem{lemma}[definition]{Lemma}
\newtheorem{corollary}[definition]{Corollary}

\begin{document}

\title{Computation of $\lambda$-classes via strata of differentials}
\author{Georgios Politopoulos}
\address{Mathematical Institute, Leiden University, PO Box 9512, 2300 RA Leiden, The Netherlands}
\email{g.politopoulos@math.leidenuniv.nl }
\author{Adrien Sauvaget}
\address{Laboratoire AGM, 2 avenue Adolphe Chauvin, 95300, Cergy-Pontoise, France}
\email{adrien.sauvaget@math.cnrs.fr}

\date{\today}

\maketitle

\begin{abstract} We introduce a new family of tautological relations of the moduli space of stable curves of genus $g$. These relations are obtained by computing the Poincaré-dual class of empty loci in the Hodge bundle. We use these relations to obtain a new expression for the Chern classes of the Hodge bundle. We prove that the $(g-i)$th class can be expressed as a linear combination of tautological classes involving only stable graphs with at most $i$ loops. In particular the top Chern class may be expressed with trees. This property was expected as a consequence of the DR/DZ equivalence conjecture by Buryak-Guéré-Rossi.
\end{abstract}

%\setcounter{tocdepth}{1}
%\tableofcontents

%%%%%%%%%%%%%%%%%%%%%%%%%%%%%%%%
%%%%%%%%%%%%%%%%%%%%%%%%%%%%%%%%
\section{Introduction}
%%%%%%%%%%%%%%%%%%%%%%%%%%%%%%%%
%%%%%%%%%%%%%%%%%%%%%%%%%%%%%%%%

\subsection{Tautological rings}

We work over $\mathbb{C}$. Let $g$ and $n$ be non-negative integers satisfying $2g-2+n>0$. We denote by 
$$\oM_{g,n} \supset \M_{g,n}^{\rm ct} \supset \M_{g,n}^{\rm rt} \supset \M_{g,n}$$
the moduli spaces of stable curves (respectively curves of compact type, curves with rational tails, and smooth curves)   of genus $g$ with $n$ markings. We recall that a stable curve is of compact type if it has only non-separating nodes, and has rational tails if it has one irreducible component of genus $g$.  Besides, we have the following families of morphisms between moduli spaces of stable curves:
\begin{itemize}
\item The {\em forgetful morphism} of the last marking: $\pi:\oM_{g,n+1}\to \oM_{g,n}$. The image of a marked curve is defined as the stabilization of the curve obtained by removing the marking with label $n+1$. For $n'\geq 1$, we will also denote by $\pi_{n'}$ the composition of $n'$ times $\pi$.
\item The {\em gluing morphism of type tree} $j_{g',I}: \oM_{g',|I|+1} \times \oM_{g-g',n-|I|+1} \to \oM_{g,n}$ defined for all $I\subset [\![ 1, n]\!]$ and $g'<g$ respecting the stability condition on both factors. The image of a pair of curves is the nodal curve obtained by identifying the last two markings.
\item The {\em gluing morphism of type loop} $j_0:\oM_{g-1,n+2}\to \oM_{g,n}$ (if $g>0$). The image of a curve is the nodal curve obtained by identifying the last two markings. 
\end{itemize}
The {\em tautological rings} $\{R^*(\oM_{g,n})\}_{g,n}$ are the family of smallest sub-$\QQ$ algebras of $A^*(\oM_{g,n},\QQ)$ stable under push-forwards along  the forgetful morphisms and gluing morphisms (see~\cite{GraPan}). Tautological classes of $\M_{g,n}, \M_{g,n}^{\rm rt}$ and $\M_{g,n}^{\rm ct}$ are defined as the restrictions of tautological classes of $\oM_{g,n}$ to the corresponding open sub-stack.  The standard examples of tautological classes are given by the two families:
\begin{itemize}
\item for all $1\leq i\leq n$, we denote by $\psi_i \in A^1(\oM_{g,n},\QQ)$ the Chern class of the cotangent line at the $i$-th marking.
\item if $m\geq 0$, we denote by $\kappa_m=\pi_*\psi_{n+1}^{m+1}\in A^m(\oM_{g,n},\QQ)$. 
\end{itemize}
\begin{definition}
If $i\in \NN$, then we denote by $R_i^*(\oM_{g,n})$ the linear subspace of $R^*(\oM_{g,n})$ spanned by push-forwards of polynomials in $\psi$ and $\kappa$-classes along gluing morphisms of type tree and at most  $i$ times the morphism of type loop. 
\end{definition} 
These subspaces form a filtration of $R^*(\oM_{g,n})$ as the tautological rings are spanned by push-forwards of polynomials in $\psi$ and $\kappa$-classes along gluing morphisms
(see the appendix of~\cite{GraPan}). The space of relations between these linear generators of $R^*(\oM_{g,n})$ is known as the set of {\em tautological relations}. Several techniques have been proposed to produce tautological relations (see~\cite{Fab},~\cite{Yin},~\cite{ClaJan} for a sample of these techniques). The largest class of known tautological relations are the Pixton-Faber-Zagier (PFZ) relations which were computed in~\cite{PanPixZvo} in cohomology and then in~\cite{Jan} in the Chow rings. An open problem is to know whether the PFZ relations recover all possible tautological relations or not. 

Here, we will produce a new family of tautological relations that we call {\em Hodge relations}. Two natural questions arise: are the Hodge relations contained in the PFZ relations? If yes, which part of the PFZ relations is recovered by the Hodge relations? Unfortunately we are unable at the moment to answer these questions.

 \subsection{Hodge classes}

We denote by $\oH_{g,n}\to \oM_{g,n}$ the {\em Hodge bundle}, i.e. the vector bundle whose fiber over a marked curve $(C,x_1,\ldots,x_n)$ is given by $H^0(C,\omega_C)$. This vector bundle is of rank $g$ and we have $\oH_{g,n}=\pi_n^*\oH_{g,0}$. We will be interested in the {\em Hodge classes} (often called {\em $\lambda$ classes}):
$$
\lambda_{i} = c_{i}(\oH_{g,n}) \in A^{i}(\oM_{g,n},\QQ) \text{ or } H^{2i}(\oM_{g,n},\QQ),
$$
and we define the {\em Hodge polynomial} as
$$
\Lambda_g(\xi)=\xi^g+\xi^{g-1}\lambda_1+\ldots+\lambda_g \in A^*(\oM_{g,n},\QQ)[\xi]
$$
In his seminal paper~\cite{Mum}, Mumford applied the Grothendieck-Riemann-Roch formula to show that these classes are tautological and expressed them in terms of $\kappa$ and $\psi$ classes and the gluing morphisms. These classes satisfy  the following remarkable properties:
\begin{itemize}
\item $\Lambda_g(1)\Lambda_g(-1)=(-1)^g $ (in particular $\lambda_g^2=0$ if $g>0$).
\item For all $i\in \NN$, and $\beta\in R^*(\oM_{g,n})\setminus R_i^*(\oM_{g,n})$, we have $\beta\cdot \lambda_{g-i}=0$.
\end{itemize}
The Hodge Polynomial plays an important role in different areas of enumerative geometry as it arises in the computation of Gromov-Witten invariants of Toric varieties via the virtual Localization formula of Graber and Pandharipande (see~\cite{GraPan1}).  

Using the theory of Double Ramification cycles, Janda-Pandharipande-Pixton-Zvonkine showed that $\lambda_g$ may be expressed as a linear combination of tautological classes constructed with  graphs with no non-separating nodes (see~\cite{JanPanPixZvo}). This result is, up to our knowledge, the only systematic expression of some Hodge classes apart from Mumford's original formula.  Here we propose an alternative expression of Hodge classes. As a consequence of this approach we obtain the following result:
\begin{theorem}\label{th:main}
For all $i\in \NN$, the class $\lambda_{g-i}$ sits in $R_i^*(\oM_{g,n})$. 
\end{theorem}
The following corollary shows that  $\lambda$ classes bound the complexity of any given tautological class. 
\begin{corollary}\label{cor:main}
For all $i \in \NN$, and all $\beta\in R^*(\oM_{g,n})$, the class $\lambda_{g-i}\cdot \beta$ sits in $R_i^*(\oM_{g,n})$.
\end{corollary}
In the specific case of $\lambda_g$, Theorem~\ref{th:main} is somewhat orthogonal to the computation via ${\rm DR}$-cycles as it implies that $\lambda_g$ may be expressed as a linear combination of decorated trees. This property of $\lambda_g$ was conjectured in~\cite{BurGueRos} as a consequence of the so-called DR/DZ equivalence conjecture as explained below. 

\subsection{Strata of differentials with prescribed zeros} The projectivization of the Hodge bundle will be denoted by  $f:\PP\oH_{g,n}\to \oM_{g,n}$,  and by $\PP\H_{g,n}$ for the restriction to $\M_{g,n}$. We recall that we have the following isomorphism:
$$
A^*(\PP\oH_{g,n},\QQ)\simeq A^*(\oM_{g,n},\QQ)[\xi]\big/ \Lambda_g(\xi),
$$
where $A^*(\oM_{g,n},\QQ)$ is included in $A^*(\PP\oH_{g,n},\QQ)$ via the pull-back along $f$ and $\xi$ is identified with the Chern class of $\mathcal{O}(1)$. We define the tautological ring of $\PP\oH_{g,n}$ as $R^*(\PP\oH_{g,n})=R^*(\oM_{g,n})[\xi]\big/ \Lambda(\xi)$. The euclidean division in $R^*(\oM_{g,n})[\xi]$ implies that any class $\alpha$ in $R^*(\oM_{g,n})[\xi]$ can be uniquely decomposed as $$
\alpha=q(\alpha) \Lambda(\xi) + r(\alpha),
$$
where $r(\alpha)$ has degree at most $(g-1)$ in $\xi$. Then $\alpha= r(\alpha)$ in  $R^*(\PP\oH_{g,n})$. 

 We fix a vector of non-negative integers $Z=(z_1,\ldots,z_n)$. We will be interested in the loci defined by:
$$
\H_g(Z)=\{(C,\omega,x_1,\ldots,x_n), \text{ s.t. ${\rm ord}_{x_i}(\omega)=z_i$ for all $1\leq i\leq n$}\} \subset \H_{g,n}.
$$
We denote by $\PP\oH_g(Z)$ the Zariski closure of $\PP\H_g(Z)$ in $\PP\oH_{g,n}$. In~\cite{Sau}, the second author proposed  an algorithm to compute a class $\alpha(g,Z)\in R^*(\PP\oH_{g,n})$ which satisfies:
$$
\alpha(g,Z)=[\PP\oH_g(Z)] \text{ in $R^*(\PP\oH_{g,n})$},
$$
where $[\,\, \cdot ]$ stands for the the Poincar\'e-dual class. Here we improve this result by showing that this algorithm uniquely determines  a class $\alpha(g, Z)$ in $R^*(\oM_{g,n})[\xi]$ whose class in $R^*(\PP\oH_{g,n})$ is equal to $[\PP\oH_g(Z)]$ (see Section~\ref{ssec:algo}). Besides, we will show that the coefficient of $\xi^i$ in $\alpha(g, Z)$ sits in $R^*_i(\oM_{g,n})$.

We will use this algorithm to compute $\alpha(g,Z)$ for $Z$ of size $(2g-1)$. In this case the locus $\PP\H_g(Z)$ is empty and thus the class $\alpha(g,Z)$ vanishes in $R^*(\PP\oH_{g,n})$. Therefore, the remainder of the euclidean division of $\alpha(Z)$ by $\Lambda(\xi)$ vanishes. This fact allows us to compute up to $g$ tautological relations in $R^*(\oM_{g,n})$. This will constitute the set of {\rm Hodge} tautological relations. In particular we will show that:
$$
\Lambda(\xi) = \frac{2^{-g+1}}{g!} \cdot \pi_{(g-1)*} \alpha\bigg(g,(1,\underset{(g-1) \times}{\underbrace{2,\ldots,2}})\bigg).
$$
We use this expression of $\Lambda(\xi)$ in $R^*(\oM_{g,1})$ to finish the proof of Theorem~\ref{th:main}. 

\subsection{Relation to the strong DR/DZ equivalence conjecture} Given a Cohomological Field Theory, one may construct two integrable systems of PDE's: the double ramification (DR) hierarchy (defined in~\cite{BurGueRos}), and the Dubrovin-Zhang (DZ) hierarchy (defined if the CohFT is semi-simple in~\cite{DubZha}). The coefficients of both hierarchies are defined by the intersection theory on $\oM_{g,n}$. 

In~\cite{BurDubGueRos} Buryak-Dubrovin-Guéré-Rossi conjectured that these two hierarchies are related by a Miura transformation, a change of coordinates in the space on which the PDE's are defined. Furthermore, in~\cite{BurGueRos} they constructed two families of classes $A^g_{z_1,\ldots,z_n}$ and $B^g_{z_1,\ldots,z_n}$ in $R^{z_1+\cdots+z_n}(\oM_{g,n})$ defined for all vectors of positive integers $Z$ of size greater than $2g-2$. They conjectured that
$$A^g_{z_1,\ldots,z_n}=B^g_{z_1,\ldots,z_n},$$
and showed that this conjecture implies the equivalence of the DR and DZ hierarchies. The case $n=1$ of this conjecture has been recently proved in~\cite{BurHerSha}. 

Besides, the class $A^g_{z_1,\ldots,z_n}$ that they constructed may be decomposed as a product $\widetilde{A}^g_{z_1,\ldots,z_n} \cdot \lambda_g$. Using this fact they conjectured that $\lambda_g$ sits in $R_0^*(\oM_{g,n})$. Based on numerical experiments we conjecture the equalities:
\begin{eqnarray*}
(2g-1)! \, \widetilde{A}^{g}_{(2g-1)}&=&\text{ the coefficient in $\xi^0$ of $q(\alpha(g,(2g-1)))$ in $R^*(\M_{g,n}^{\rm ct})$}.\\
(2g-1)! \, {B}^{g}_{(2g-1)}&=&\text{ the coefficient in $\xi^0$ of $\alpha(g,(2g-1))$ in $R^*(\oM_{g,n})$}.
\end{eqnarray*}
This conjecture would allow us to reprove the case $n=1$ of the $A=B$ conjecture. For more general values of $Z$, such a simple statement does not holds but we expect that the $DR/DZ$ correspondence follows from a linear combination of Hodge relations.

\subsection*{Acknowledgement} We would like to thank Sacha Buryak, David Holmes, Jérémy Guéré, Reinier Kramer, Scott Mullane, Paolo Rossi, Johannes Schmitt, and Sergey Shadrin for useful discussions on the topic. 

%%%%%%%%%%%%%%%%%
\section{Inductive computation of classes of strata of differentials} 
%%%%%%%%%%%%%%%%%

Here we recall the computation of the function $\alpha$ defined in the introduction.

\subsection{Strata algebra} Given a genus $g$ and a finite set $I$, a {\em stable graph} of genus $g$ marked by $I$ is the datum of $$\Gamma=(V,H,g:V\to \mathbb{Z}_{\geq 0},\iota:H\to H,\phi:H\to V,H^\iota\simeq I),$$
where:
\begin{itemize}[leftmargin=25pt]
\item An element $v\in V$ is a {\em vertex}. The value $g(v)$ the {\em genus} of $v$.
\item An element $h\in H$ is called a {\em half-edge}. We say that it is incident to $\phi(h)$, and write $h\mapsto v$ if $\phi(h)=v$. Moreover, we denote by $n(v)$ the {\em valency} of the vertex $v$, that is the of half-edges incident to $v$. 
\item The function $\iota$ is an involution. The cycles of length 2 will be denoted by $E$ and are called the {\em edges}.  
\item The fixed points of $\iota$ are called {\em legs}. 
\item For all vertices $v$ we have the stability condition $2g(v)-2+n(v)>0$.
\item The graph~$(V,E)$ is connected.
\item The \emph{genus} of~$\Gamma$ is defined as
\[
g(\Gamma) = h^1(\Gamma)+\sum_{v\in V} g(v),\quad \text{ with } \quad h^1(\Gamma)=|E|-|V|+1.
\] the number $h^1(\Gamma)$ is the {\em number of loops} of $\Gamma$. We impose that $g(\Gamma)=g$.
\end{itemize}
{An \emph{automorphism} of~$\Gamma$ consists of automorphisms of the sets~$V$ and~$H$ that leave invariant the data $g,\iota$ and $\phi$}. A stable graph is said to be \emph{of compact type} if $h^1(\Gamma)=0$, i.e.\ if the graph is a tree. 

We write by ${\rm Stab}_{g,I}$ the set of stable graphs of genus $g$ and marked by $I$. We simply denote ${\rm Stab}_{g,n}$ if $I=\{1,\ldots,n\}$.

Given a stable graph $\Gamma$ in ${\rm Stab}_{g,n}$, we denote by:
$$
\oM_{\Gamma}=\prod_{v\in V} \oM_{g(v),n(v)}
$$
the associated moduli space. If there is a canonical morphism:
$$
\zeta_{\Gamma} \to \oM_{g,n}
$$
defined by a composition of gluing morphisms of type loop or tree. The degree of this morphism on its image is $|{\rm Aut}(\Gamma)|$. 

A {\em decorated graph} is the datum of $\Gamma$ and $c(v)$ a product of $\kappa$ and $\psi$ classes in $\oM_{g(v),n(v)}$ and of degree at most the dimension of $\oM_{g(v),n(v)}$, for all vertices of $v$ of $\Gamma$. We denote by $
\mathcal{S}_{g,n}$ the vector space whose is given by all decorated graphs. This vector space is naturally equipped with a structure of graded algebra and we call it the {\em strata algebra}. By~\cite{GraPan}, there is a natural surjective morphism of graded algebra
s:
$$
\mathcal{S}_{g,n}^* \to R^*(\oM_{g,n}) \to 0. 
$$
This morphism is defined by $$(\Gamma,\{c(v)\}_{v\in V})\mapsto \zeta_{\Gamma *}\left(\bigotimes_{v\in V} c(v)\right).$$
The kernel of this morphism is the space of tautological relations. 

As we consider polynomials in $R^*(\oM_{g,n})[\xi]$, we extend the  notation for $\zeta_{\Gamma*}$:
$$
\zeta_{\Gamma *} : \left(\bigotimes_{v\in V} R^*(\oM_{g(v),n(v)})[\xi]\right) \simeq \left(\bigotimes_{v\in V} R^*(\oM_{g(v),n(v)})\right)[\xi]\to R^*(\oM_{g,n})[\xi],
$$
and we extend the definition of the push-forwards and pull-back along the forgetful morphisms to $R^*(\oM_{g,n})[\xi]$ in the same way.

Finally we recall that, if a class sits in $\left(\bigotimes_{v\in V'} R^*(\oM_{g(v),n(v)})[\xi]\right)$ for some subset $V'\subset V$ then it naturally defines a class in  $\left(\bigotimes_{v\in V} R^*(\oM_{g(v),n(v)})[\xi]\right)$ by extending by $1$ on each of the component indexed by $V\setminus V'$.

%%%%%%%%%%%
\subsection{Generalized strata with residue conditions}\label{ssec:res}
%%%%%%%%%%%

Now, we fix a triplet of vectors of non-negative integers of length $k$: $\underline{g}=(g_1,\ldots,g_k), \underline{n}=(n_1,\ldots, n_k)$, and $\underline{m}=(m_1,\ldots, m_k)$ which is {\em stable},  i.e. which satisfies $2g_i-2+n_i+m_i>0$ for all $1\leq i\leq k$. Then we denote by 
$$
\oM_{\underline{g},\underline{n},\underline{m}}=\prod_{i=1}^k \oM_{g_i,n_i+m_i}, \text{ and } \M_{\underline{g},\underline{n},\underline{m}}=\prod_{i=1}^k \M_{g_i,n_i+m_i}.
$$
 Then we fix 
\begin{eqnarray*}
\underline{P}&=&(P_1=(p_{1,1}\ldots,p_{1,m_1}),\ldots,P_k=(p_{k,1}\ldots,p_{k,m_k})),\\
\text{ and } \underline{Z}&=&(Z_1=(z_{1,1}\ldots,z_{1,n_1}),\ldots,Z_k=(z_{k,1}\ldots,z_{k,n_k})),
\end{eqnarray*}
vectors of vectors of positive (respectively non-negative integers). Then we denote by $\oH_{\underline{g},\underline{n}}(\underline{P})\to  \oM_{g_i,n_i+m_i})$ the vector bundle whose fiber over a marked curve $((C_i, x_{i,1},\ldots x_{i,n_i+m_i})_{1\leq i\leq k}$ is canonically identified with:
$$
\bigoplus_{i=1}^k H^0(C,\omega_C(p_{i,1} x_{i,n_i+1}+\ldots + p_{i,n_i+m_i}).
$$
We denote by $\H_{\underline{g},\underline{n}}(\underline{P})$ the restriction of this bundle to $ \M_{\underline{g},\underline{n},\underline{m}}$ and by 
$$
\H(\underline{g},\underline{P},\underline{Z}) \subset \H_{\underline{g},\underline{n}}(\underline{P})
$$
the locus of differentials satisfying ${\rm ord}_{x_{i,j}}(\omega)=z_{i,j}$ for all $1\leq i\leq k$ and $1\leq j\leq n_i$.

We denote by ${\rm Res}(\underline{m})$ the sub-vector space of  $\bigoplus_{i=1}^k \CC^{m_i}$ of vectors $(r_{i,j})_{\begin{smallmatrix} 1\leq i\leq k \\ 1\leq j\leq m_i\end{smallmatrix}}$ satisfying:
$$
r_{i,1}+\ldots + r_{i,m_i}=0
$$
for all $1\leq i\leq k$. Then, if $\mathfrak{R}$ is a sub-vector space of ${\rm Res}(\underline{m})$, we denote by $\H(\underline{g},\underline{P},\underline{Z},\mathfrak{R})$ the subspace of $\H(\underline{g},\underline{P},\underline{Z})$ of differentials with residues at the marked poles sitting in $R$. Finally we denote by $\PP\oH(\underline{g},\underline{P},\underline{Z}, \mathfrak{R})$ the closure of the projectivization of this space in $\PP\oH_{\underline{g},\underline{n}}(\underline{P}).$

%%%%%%%%%%%
\subsection{Level graphs} 
%%%%%%%%%%%

A {\em twist} on a stable graph $\Gamma$ %in ${\rm Stab}_{g,n}$  
is a function $\mu:H\to {\ZZ}$ satisfying the following conditions:
\begin{itemize}[leftmargin=25pt]
\item For all $v\in V$, we denote by $\mu(v)$ the vector of twists at half-edges incident to $v$. If $n(v)=1$, then we impose that $|\mu(v)|\ \leq 2g(v)-2$.
\item   If $e=(h,h')$ is an edge of $\Gamma$ from $v$ to $v'$, then we have $\mu(h)=-\mu(h')-2$. Moreover we denote by $\mu(e)=|\mu(h)+1|$ (for any of the two half-edges).
\item There exists a partial order~$\geq$ on~$V$ such that for all vertices $v,v'$ connected by an edge $(h,h')$ we have  $(v\geq v') \Leftrightarrow  (\mu(h)+1\geq 0)$.
\end{itemize}

A \emph{twisted graph}  is a pair $(\Gamma,\mu)$, where $m$ is a twist on~$\Gamma$. Given a vector of non-negative integers $Z=(z_1,\ldots,z_n)$, %and a vector $P=(p_1,\ldots,p_n)$ of positive integers, 
then the twisted graph is said {\em compatible} with $Z$, if the twist at the $i$th leg is equal to $z_i$ while the twist at $n+i$th leg is equal to $-p_i$. 

\begin{definition} Let $(\underline{g},\underline{Z})$ ($\underline{P}$ is empty here) be as in the previous section. A level graph $\overline{\Gamma}=((\Gamma_i,\mu_i)_{1\leq i\leq k},\ell)$ of depth $d$ and compatible with $\underline{Z}$ is the datum of a twisted graph of genus $g_i$ with $n_i$ marking and compatible $Z_i$ for all $1\leq i\leq k$, and a surjective function 
$$
\ell: \cup_{1\leq i\leq k}  V(\Gamma_i) \to \{0,-1,\ldots,-d\}
$$ 
satisfying: $\ell(v)<\ell(v') \Leftrightarrow v<v'$ for all pairs of vertices $(v,v')$, and if $v$ is of level smaller than 0 then $|\mu(v)|\leq 2g(v)-2$. 

We denote by ${\rm LG}_d(\underline{g},\underline{Z})$ the set of level graphs $\overline{\Gamma}=(\Gamma,\mu,\ell)$, i.e. twisted graphs compatible with $\underline{Z}$, with a level function of depth $d$, which is surjective and which has no horizontal edges. Finally, we denote by:
 $$m(\overline{\Gamma})= \prod_{e \in \cup_i E(\Gamma_i)} \mu(e).$$
 
Given a level graph, we denote  $h^1(\overline{\Gamma})$ stands for the sum of the $h^1(\Gamma_i)$ while the automorphism group of $\overline{\Gamma}$ is the product of the groups of automorphism of $\Gamma_i$ that commute with the twist and level functions.
\end{definition}

\begin{definition} A level graph is $\overline{\Gamma}$ in ${\rm LG}_1(\underline{g},\underline{Z})$ is a {\em bi-colored graph} if there is some $1\leq i_0\leq k$ such that for all $i\neq i_0$, the graph $\Gamma_i$ is trivial and the level function maps the unique vertex of this graph to 0.

We denote by ${\rm Bic}(\underline{g},  \underline{Z})$ the set of such graphs, and if $1\leq i\leq n$ and $1\leq j\leq n_i$, we denote by ${\rm Bic}(\underline{g},  \underline{Z})_{i,j}$ the set of such graph such that the marking $x_{i,j}$ sits on the level $-1$. Besides, if $(i',j')$ is another marking then we denote by ${\rm Bic}(\underline{g},  \underline{Z})_{i,j,i',j'}$ with both $x_{i,j}$ and $x_{i',j'}$ at level $-1$, while ${\rm Bic}(\underline{g},  \underline{Z})_{i,j}^{i',j'}\subset {\rm Bic}(\underline{g},  \underline{Z})_{i,j}$ is the set of bi-colored graphs with $x_{i,j}$ at level $0$.
\end{definition}

\subsection{Strata associated to a level graph}

Let $\overline{\Gamma}$ be a level graph in ${\rm LG}_d(\underline{g},\underline{Z})$. For all $0\geq i \geq -d$, $\overline{\Gamma}$ determines a stable triple $(\underline{g}^{[i]}, \underline{n}^{[i]},\underline{m}^{[i]})$  as in the previous section: the length of these vectors is the cardinality of $\ell^{-1}(i)$, $\underline{g}^{[i]}$ is the vector of genera of vertices in $\ell^{-1}(i)$, $\underline{n}^{[i]}$ is the vector of numbers of half-edge with non-negative twists, while $\underline{m}^{[i]}$ is the vector of numbers of half-edge with negative twists. Moreover, $\overline{\Gamma}$ determines $\underline{P}^{[i]}$ and $\underline{Z}^{[i]}$  the vectors of  opposite of negative twists, and non-negative twists respectively at half-edges incident to vertices in $\ell^{-1}(i)$. 

\begin{remark}
The level graph $\overline{\Gamma}$ does not exactly determine these vectors, as we need to make a choice of an ordering on the set of vertices of level $-i$ and on the half-edges incident to vertices of level $-i$. We will have to carefully show that all such choices are equivalent in our computations.
\end{remark}

Finally we determine spaces of residue conditions $\mathfrak{R}^{[i]} \subset {\rm Res}(\underline{m}^{[i]})$ for all $0\geq i\geq -d$ (see Section~\ref{ssec:res} for the definition of residue conditions) as follows:
\begin{itemize}
\item $\mathfrak{R}^{[0]}$ is trivial (there are no negative twists at vertices of level 0).
\item $\mathfrak{R}^{[-1]}$ is defined by the following conditions: if $v$ is a vertex in $\ell^{-1}(0)$ then
$$
\sum{(h,h')\in E(\Gamma), \text{ s.t. } (h'\mapsto v, h\mapsto \ell^{-1}(-1))} r_h = 0.
$$
Here the summation is over all the edges with an extremity incident to $v$ while the other one maps to a vertex in $\ell^{-1}(-1)$, and $r_h$ stands for the residue associated to the label of the half-edge incident to the level $-1$.
\item For $i<-1$, we first construct the graph $\overline{\Gamma}'$ obtained from $\overline{\Gamma}$  by contracting all edges between levels greater than $i$. Then the vertices of $\overline{\Gamma}$  of level $-i$ are in correspondence with vertices of $\overline{\Gamma}'$ of level $-1$, and we define $\mathfrak{R}^{[i]}$ as the vector space $\mathfrak{R}^{[-1]}$ of this newly constructed graph.
\end{itemize}

With this notation at hand we let $\PP\oH(\overline{\Gamma})^{[i]}=\PP\oH(\underline{g}^{[i]}, \underline{P}^{[i]},\underline{Z}^{[i]},\mathfrak{R}^{[i]})$ and denote by $f^{[i]}: \PP\oH(\underline{g}^{[i]}, \underline{P}^{[i]}) \to  \oM_{\underline{g}^{[i]},\underline{n}^{[i]},\underline{m}^{[i]}}$ the forgetful morphism of the differential. %The stratum associated to $\overline{\Gamma}$ is then defined as:
%$$
%\PP\oH(\Gamma)=\PP\oH(\overline{\Gamma})^{[0]} \times \prod_{-d\leq i<0} f_i(\PP\oH(\overline{\Gamma})^{[i]}).
%$$
%It canonically sits in the boundary of $\PP\oH(g,Z)$ via the morphism: $\zeta_{\overline{\Gamma}}: \PP\oH(\Gamma)\to \PP\oH(g,Z)$ which associates to an element of $\PP\oH(\Gamma)$ the differential which may be non-zero on the vertices of level 0 and vanishes identically on the others.

%%%%%%%%%%
\subsection{Inductive computation of classes of strata}\label{ssec:algo}
%%%%%%%%%%

For all pairs $(\underline{g},\underline{Z})$ we define a class $\alpha(\underline{g},\underline{Z})\in \bigotimes_{i=1}^k R^*(\oM_{g_i,n_i})[\xi]$ of degree $|\underline{Z}|$ inductively as follows:
\begin{enumerate}
\item \underline{\textit{Base case.}} If all entries of $\underline{Z}$ are equal to 0 then $\alpha(\underline{g},\underline{Z})=1$.
\item \underline{\textit{Incrementing $(i,j)$.}} Otherwise, we chose an entry $1\leq i\leq j$ and $1\leq j\leq n_i$ such that $z_{i,j}>0$. We denote by $\underline{Z}_{i,j}$ the vector obtained from $\underline{Z}$ by diminishing $z_{i,j}$ by $1$. Then we set:
$$
\alpha(\underline{g},\underline{Z})= (\xi+z_{i,j} \psi_{i,j}) \alpha(\underline{g},\underline{Z}_{i,j})- \sum_{\overline{\Gamma}\in {\rm Bic}(\underline{g},\underline{Z})_{i,j}} m(\overline{\Gamma}) \alpha(\overline{\Gamma}),
$$
where $\alpha(\overline{\Gamma})$ is defined by the next point. 
\item \underline{\textit{Class of a level graph.} } If $\overline{\Gamma}$ is a level graph, then we define:
$$
\alpha(\overline{\Gamma})=\frac{ \xi^{h^1(\overline{\Gamma})}}{|{\rm Aut}(\overline{\Gamma})|} \zeta_{\Gamma *}\left(\bigotimes_{-d\leq i\leq -1} f_{i*}[\PP\oH(\overline{\Gamma})^{[i]}] \cdot \left(  \bigotimes_{v\in \ell^{-1}(i)} \Lambda_v(\xi)\right) \right) \otimes \alpha(\underline{g}^{[0]},\underline{Z}^{[0]})  ,
$$
where for all $v\in  \ell^{-1}(-1)$,  $ \Lambda_v(\xi)$ stands for the class $(\xi^{g(v)}+ \ldots+\lambda_{g(v)})$ in  $R^*(\oM_{g(v),n(v)})[\xi]$.\end{enumerate}

This algorithm is a priori ill-defined as, at the second point, one has to make a choice of value of $(i,j)$ to increment.  However, we should remark that the function $\alpha$ is defined by induction on $|\underline{Z}|$. The first point provides the initialization of this algorithm, while the second one provides the induction:   the size of $\underline{Z}_{i,j}$ is equal to $|\underline{Z}|-1$, while the class $\alpha(\overline{\Gamma})$ of a bi-colored graph in the sum of the right-hand side vanishes if $\underline{Z}^{[0]}$ is of size at least $|\underline{Z}|$. This is due to the vanishing of $f_{-1*}[\PP\oH(\overline{\Gamma})^{[-1]}]$ for dimension reasons.

\begin{lemma}
The class  $\alpha(\underline{g},\underline{Z})$ is uniquely determined by the above algorithm (i.e. independently of the choices at the second and third points).  
\end{lemma}

\begin{proof} We prove this lemma by induction on the size of $\underline{Z}$, i.e. $\sum_{i,j}^k z_{i,j}$. The initialization is uniquely determined as $\alpha(\underline{g},\underline{Z})=1$ if the size of $\underline{Z}$ is 0. 

Now we assume that the size of $\underline{Z}$ is positive.
% First, we need to remark that the bi-colored $\overline{\Gamma}$ such $\alpha(\overline{\Gamma})$ is non-trivial satisfy that $\underline{Z}^{[0]}$ is of size smaller than $\underline{Z}$, indeed otherwise $f_{-1*}[\PP\oH(\overline{\Gamma}_{-1})]=0$ for dimension reasons. 
Therefore if only one entry of $\underline{Z}$ is non-zero then $\alpha(\underline{g},\underline{Z})$ is uniquely determined. Indeed, for each bi-colored graph, if the graph contributes non-trivially then $\underline{Z}^{[0]}$ is of size smaller than $\underline{Z}$ and thus the class of this bi-colored graph is uniquely determined. 

Therefore we assume that two entries $z_{i,j}$ and $z_{i',j'}$ are non-trivial. We denote by $\widetilde{\underline{Z}}$ the vector obtained from $\underline{Z}$ by diminishing $z_{i,j}$ and $z_{i',j'}$ by one. We denote by  $\alpha(\underline{g},\underline{Z})$ obtained from $\alpha(\underline{g},{\underline{Z}}_{i,j,i',j'})$  by incrementing $z_{i',j'}-1$ to $z_{i',j'}$ and then to $z_{i,j}-1$ to $z_{i,j}$ while $\alpha(\underline{g},\underline{Z})'$ is the converse. We will show that $\alpha(\underline{g},\underline{Z})=\alpha(\underline{g},\underline{Z})'$. If $i\neq i'$ then the resulting formulas are equivalent, thus we will assume that $i=i'$ for the rest of the proof. Then we have the following expression:
\begin{eqnarray*}
\alpha(\underline{g},\underline{Z}) &=& (\xi+z_{i,j'} \psi_{i,j'}) (\xi+z_{i,j} \psi_{i,j}) \alpha(\underline{g},\widetilde{\underline{Z}}) 
-\!\!\!\!\!\!\!\!\! \sum_{\overline{\Gamma}\in {\rm Bic}(\underline{g},\underline{Z}_{i,j'})_{i,j'}}   \!\!\!\!\!\!m(\overline{\Gamma})\cdot \alpha(\overline{\Gamma}) \\
\\&& - \sum_{\overline{\Gamma}\in {\rm Bic}(\underline{g}, \widetilde{\underline{Z}})_{i,j}}  m(\overline{\Gamma}) \cdot  (\xi+z_{i,j'} \psi_{i,j'})\cdot \alpha(\overline{\Gamma}).
\end{eqnarray*}
Then we introduce the following notation:
\begin{eqnarray*}
T_{1}^{j}&=& \sum_{\overline{\Gamma}\in {\rm Bic}(\underline{g}, \widetilde{\underline{Z}})_{i,j}^{i,j'}}  m(\overline{\Gamma}) \cdot  (\xi+z_{i,j'} \psi_{i,j'}) \cdot \alpha(\overline{\Gamma}) -\!\!\!\!\!\!\!\!\! \sum_{\overline{\Gamma}\in {\rm Bic}(\underline{g},\underline{Z}_{i,j})_{i,j}^{i,j'}}   \!\!\!\!\!\!m(\overline{\Gamma})\cdot \alpha(\overline{\Gamma})
\\
T_{1}^{j'}&=& \sum_{\overline{\Gamma}\in {\rm Bic}(\underline{g}, \widetilde{\underline{Z}})_{i,j'}^{i,j}}  m(\overline{\Gamma}) \cdot  (\xi+z_{i,j} \psi_{i,j}) \cdot \alpha(\overline{\Gamma}) -\!\!\!\!\!\!\!\!\! \sum_{\overline{\Gamma}\in {\rm Bic}(\underline{g},\underline{Z}_{i,j'})_{i,j'}^{i,j}}   \!\!\!\!\!\!m(\overline{\Gamma})\cdot \alpha(\overline{\Gamma})
\\
T_{2} &=&  \sum_{\overline{\Gamma}\in {\rm Bic}(\underline{g},\underline{Z}_{i,j'})_{i,j,i,j'}}   \!\!\!\!\!\!m(\overline{\Gamma})\cdot \alpha(\overline{\Gamma}) -\!\!\!\!\!\!\!\!\! \sum_{\overline{\Gamma}\in {\rm Bic}(\underline{g},\underline{Z}_{i,j})_{i,j,i,j'}}   \!\!\!\!\!\!m(\overline{\Gamma})\cdot \alpha(\overline{\Gamma}) \\
T_{3} &=& \sum_{\overline{\Gamma}\in {\rm Bic}(\underline{g}, \widetilde{\underline{Z}})_{i,j,i,j'}}  m(\overline{\Gamma}) \cdot  (z_{i,j'} \psi_{i,j'}-z_{i,j} \psi_{i,j})\cdot \alpha(\overline{\Gamma}) 
\end{eqnarray*}
Then with these notation we have
\begin{eqnarray*}
\alpha(\underline{g},\underline{Z})- \alpha(\underline{g},\underline{Z})' = T_1^{j'}-T_{1}^{j}- T_2- T_3.
\end{eqnarray*}
In order to finish the proof, we will show that $T_1^{j'}-T_{1}^{j}=T_2+T_3$ are both are equal to
$$
T\coloneqq \sum_{\overline{\Gamma}\in {\rm Tri}(\underline{g}, \widetilde{\underline{Z}})_{i,j'}^{i,j} }m(\overline{\Gamma}) \alpha(\overline{\Gamma})-\sum_{\overline{\Gamma}\in {\rm Tri}(\underline{g}, \widetilde{\underline{Z}})_{i,j}^{i,j'}} m(\overline{\Gamma}) \alpha(\overline{\Gamma}),    
$$

where ${\rm Tri}(\underline{g}, \widetilde{\underline{Z}})_{i,j}^{i,j'}$ is the set of Tri-colored graphs, i.e. level graphs of depth $2$ with no horizontal edges such that: only the graph component $i$ is non-trivial (the others are trivial graphs of level 0), and the leg $(i,j)$ is incident to the level $-2$ while the leg $(i',j')$ is incident to the level $-1$. 

To prove that $T= T_1^{j'}-T_{1}^{j}$, we simply apply the induction formula to the level 0 part of each graph of ${\rm Bic}(\underline{g}, \widetilde{\underline{Z}})_{i,j'}^{i,j}$. Multiplying the class of the graph by $(\xi+z_{i,j} \psi_{i,j})$ has the effect of either incrementing $z_{i,j}-1$ to $z_{i,j}$ or to create a new intermediate level that carries the leg $(i,j)$, i.e. an element of ${\rm Tri}(\underline{g}, \widetilde{\underline{Z}})_{i,j}^{i,j'}$. The same 
holds if we replace $j$ by $j'$ thus the equality. 

To prove that $T=T_2+T_3$, we use a similar idea, but this at the level of 
the moduli space of curves. Indeed, here we only need to compare the classes appearing at negative levels. Applying the splitting formula of~\cite{Sau}, Theorem~6, to compute the intersection of $(z_{i,j'} \psi_{i,j'}-z_{i,j} \psi_{i,j})$ with the class of a graph in ${\rm Bic}(\underline{g}, \widetilde{\underline{Z}})_{i,j,i,j'}$ we may express it as a sum on graphs with one more level and with the legs $(i,j)$ and $(i,j')$ on distinct levels (this is the term $T_2$). 
\end{proof}

\subsection{Properties of $\alpha$}
In order to prove Theorem~\ref{th:main}, we will prove the following properties of the function $\alpha$.  
\begin{proposition}\label{pr:prop1}
For all $(g,Z)$ we have: $\alpha(g,z)\simeq [\PP\oH(g,Z)]$ in $R^*(\PP\oH_{g,n})$.
\end{proposition}
The proposition follows from the main result of~\cite{Sau}. 
\begin{proposition}\label{pr:prop2}
If $I\subset \{1,\ldots,n\}$ is a set of size at least 2, then we denote by $\delta_I$ the class of the boundary divisor of $\oM_{g,n}$ of curves with a node separating a component of genus 0 with the markings in $I$ from the rest of the curve. 
$$
\delta_I \cdot \alpha(g,Z)= j_{0,I *}(\alpha(g, Z_I)\otimes 1) \text{ in $R^*(\M_{g,n}^{\rm rt})[\xi]$}
$$
where $Z_I$ is the vector with entries: the $z_i$  for $i\notin I$ and $z_I=\sum_{i\in I} z_i$ (and we recall that $j_{0,I}:\oM_{g,n-|I|+1}\times \oM_{0,|I|+1}\to \oM_{g,n}$ is the gluing morphism).
\end{proposition}

\begin{proof}  We will prove this result by induction on the size of $Z$.  For the base case, we chose any set $I$ of size at least $2$. Then  we have ${j_{0,I}^*}\alpha(g,(0,\dots,0))= {{j_{0,I}^*}}1=1\otimes 1$. Note that this is equal to
$\alpha(g,(0,\dots,0))\otimes 1$.

Suppose that the inductive hypothesis holds. We fix some vector $Z$.  We compute the class $\alpha(g,(z_1+1,z_2,\ldots,z_n))$ by applying the inductive formula of the previous section. 
    \begin{equation}\label{eq1}
        \alpha(g,(z_1+1,z_2,\ldots))=(\xi+(z_1+1)\psi_1)\alpha(g,Z) - \sum_{\overline{\Gamma}\in {\rm Bic}({g},{Z})_1} m(\overline{\Gamma}) \alpha(\overline{\Gamma}).
    \end{equation}
    As we only compute the restriction of $\alpha(g,(z_1+1,z_2,\ldots))$ to $R^*(\M_{g,n}^{\rm rt})[\xi]$, the only bi-colored graphs that contribute are the ones with one vertex of level $-1$ of genus $0$ and with one vertex of level $0$. By induction, the class of such a graph is $\delta_{0,I'} \cdot \alpha(g,Z)$ where $I'$ is the set of markings on the vertex of level $-1$. Thus equation~\eqref{eq1} becomes:
    \begin{equation}\label{eq1}
        \alpha(g,Z)=(\xi+z_1\psi_1)\alpha(g,Z)  - \alpha(g,Z) \cdot\left(\sum_{\begin{smallmatrix} 1\in I\subseteq[1,n] \\ |I|>1 
        \end{smallmatrix}}(z_{I'}+1)\cdot \delta_{I'}\right) 
    \end{equation}
    To calculate the pullback of (\ref{eq1}) under ${j_{0,I}^*}$ we treat two cases, namely, $1\in I$ or 
    $1\notin I$. This is because ${j_{0,I}^*}(\delta_{I'})=0$ unless one of the following holds:  $I'\subset I$ or $I\subset I'$. \bigskip
    
 \noindent   \underline{\textit{Case ${1\notin I}$.}} In this case we cannot have $I'\subset I$, thus we
    only need to calculate ${j_{0,I}^*}(\cdot\delta_{I'})$ only in the case $I\subsetneq I'$. In this case, we have $\delta_{I}\cdot \delta_{I'}=j_{0,I *}(\delta_{\{h\} \cup I\setminus I'}\otimes 1)$, where $h$ is the half-edge on the component of genus $g$ of the graph defining the class $\delta_I$. Therefore, we may replace the summation over $I'$ in equation~\eqref{eq1} by a summation over $I''=I'\setminus I$. Then, using the induction assumption we get:
     \begin{eqnarray*}
        {j_{0,I}^*}\left(\alpha(g,(z_1+1,z_2,\ldots)\right)&=& \big((\xi+z_1\psi_1)\alpha(g,Z_I)\big) \otimes 1\\
        &&  - \sum_{\begin{smallmatrix} I''\subseteq[1,n]\setminus I \\ |I''|>0 
        \end{smallmatrix}} z_{I\cup I''} \big(\alpha(g,Z_{I})\cdot \delta_{I''\cup \{h\}}\big) \otimes 1  \\
        &=& \alpha(g,(z_1+1,\{z_i\}_{i\notin I}, z_I)) \otimes 1,
    \end{eqnarray*}
    Thus ending the induction.\bigskip
    
\noindent   \underline{\textit{Case $1\in I$.}} Here, all three cases $I=I'$, $I\subsetneq I'$ and $I'\subsetneq I$  should be considered. Firstly, if $I=I'$, we have the classical formula for the auto-intersection of $\delta_I$:
$$
j_{0,I}^{*}(\delta_I)=-\psi_h \otimes 1-1 \otimes \psi_{h'},
$$
where $h$ and $h'$ are the two half-edges of the unique edge of the graph defining $\delta_I$. Besides the cases $I\subsetneq I'$ and $I'\subsetneq I$ can be treated as above, thus we obtain:
\begin{eqnarray*}
        {j_{0,I}^*} \,\alpha(z_1+1,z_2,\dots))&=& \beta_1 \otimes 1 + \alpha(g,z_I) \otimes \beta_2,%\\
         %\Bigg( \big((\xi+(z_I+1)\psi_h) \cdot\alpha(g,Z_I)\big)\otimes 1  \\
        %&&- \sum_{ I''\subseteq[1,n]\setminus I} \Big(1+z_I+\sum_{i\in I''} z_i\Big) \cdot\big( \alpha(g,Z_I)\delta_{I''\cup \{h\}}\big ) \otimes 1\Bigg)  \\
        %&&+ \alpha(g,Z_I) \otimes
        %\left((z_1+1)\psi_1+(z_I+1)\psi_{h'}-\!\!\!\sum_{1\in I'\subset I}  (z_{I'}+1)\delta_{I'} \right)\\
        %&=& \alpha(g,(\{z_i\}_{i\notin I},z_I+1)) \otimes 1\\
        %&& + \alpha(g,Z_I) \otimes (0). 
\end{eqnarray*}
where $\beta_1$ and $\beta_2$ are given by: 
\begin{eqnarray*}
        \beta_1&=& \bigg((\xi+(z_I+1)\psi_h) - \sum_{ I''\subseteq[1,n]\setminus I} \Big(1+z_I+\sum_{i\in I''} z_i\Big) \cdot \delta_{I''\cup \{h\}}\big ) \Bigg) \cdot \alpha(g,Z_I)\\
        &=& \alpha(g,(\{z_i\}_{i\notin I},z_I+1)), \text{ and}\\
        \beta_2&=& (z_1+1)\psi_1+(z_I+1)\psi_{h'}-\!\!\!\sum_{1\in I'\subset I}  (z_{I'}+1)\delta_{I'} = 0.
\end{eqnarray*}
For $\beta_1$, we went from the first line to the second by using the formula for the incrementation of the value at a half-edge. For the term $\beta_2$ we have used the splitting formula of~\cite{Sau}, Theorem~6. This finishes the induction.
\end{proof}

%%%%%%%%%%%%%%%%
\section{Computation of Hodge classes}
%%%%%%%%%%%%%%%%

In order to prove Theorem~\ref{th:main}, we begin by stating several properties of the filtration of the tautological rings introduced in the introduction.

\begin{lemma}\label{lem:fil1}
If $\beta \in R_i^*(\oM_{g,n})$, then:
\begin{enumerate}
\item $\pi^*\beta$, $\pi_*\beta$ are in $R^*_i(\oM_{g,n+1})$ and $R^*_i(\oM_{g,n-1})$ respectively;
\item if $\beta=\pi^*\beta'$ then $\beta'\in R^*_i(\oM_{g,n-1})$;
\item if  we assume that $\lambda_{g'-i}\in R_i^*(\oM_{g',1})$ for all $1\leq g'\leq g,$ and $i\in \NN$, then we have $\lambda_{g-i}\cdot \beta\in R_i^*(\oM_{g,n})$. 
\end{enumerate}
\end{lemma}
\begin{proof}
The first property is obvious from the construction of the strata algebra. To show that the second one holds,  we use the fact that $\psi_{n}\beta$ sits in $R_i^*(\oM_{g,n})$ and thus $\beta'= (2g-3+n)^{-1}\pi_{*}(\psi_{n}\beta)$ sits in $R_i^*(\oM_{g,n-1})$. 

In order to prove the last point, we assume that $\lambda_{g'-i}\in R_i^*(\oM_{g',1})$ for all $1\leq g'\leq g,$ and $i\in \NN$.  We first remark that the first two point implies that for any choice of non-negative integers $g',n',i$ such that $2g'-2+n'>0$ and $g'\leq g$, then $\lambda_{g'-i}$ sits in $R_i^*(\oM_{g,n})$. Indeed, this follows from the fact that the  $\lambda$ classes are pull-back from $\oM_{g}$ if $g\geq 2$  or from $\oM_{1,1}$ (and are trivial in genus 0). 

Then, let $\beta$ be a class of the form $\zeta_{\Gamma*}(P)$, where  $\Gamma$ is a stable graph of genus $g'\leq g$ and with $n$ markings, and $P$ is a monomial in $\psi$ and $\kappa$ classes. We fix some value of $i$. Moreover, we chose any order on the vertices of $\Gamma$, i.e. an identification $V\simeq \{v_1,\ldots,v_k\}$. Then, we have the following identity:
$$
P\zeta_{\Gamma}^*\lambda_{g-i}=  \sum_{i_1+\ldots+i_k=i-h^{1}(\Gamma)} P\cdot\bigotimes_{j=1}^k \lambda_{g(v_j)-i_j}.
$$ 
Where the $i_j$'s in the sum are non-negative integers. Moreover, for any partition $i_1,\ldots, i_k$ in this sum. By assumption, each of the $\lambda_{g(v_j)-i_j}$ sits in  $R^*_{i_j}(\oM_{g(v),n(v)})$. Thus,  $\beta \cdot  \lambda_{g'-i}$ is a linear combination of decorated graphs with at most $i$ loops.
\end{proof}

%We set the following convention:
%\begin{notation}
%For all $i\geq 0$, we say that a class $\alpha$ in $R^*(\oM_{g,n})$ satisfies the property ${\rm Loop}_i$ if it may be expressed as a linear combination of classes in the strata algebra defined by graphs with at most $i$ loops.
%\end{notation}
%\begin{lemma} A class $\beta$ satisfies ${\rm Loop}_i$ if and only if its pull-back along the forgetful morphism of a point satisfies ${\rm Loop}_i$. 
%\end{lemma}
%\begin{proof}
%We denote $\beta'=\pi^*(\beta)$. It $\beta'$ satisfies ${\rm Loop}_i$ then $\psi_{n+1}\beta'$ satisfies ${\rm Loop}_i$ and thus $\beta= (2g-2+n)^{-1}\pi_{*}(\psi_{n+1}\beta')$ satisfies ${\rm Loop}_i$ too. The converse is straightforward.
%\end{proof}

We will prove Theorem~\ref{th:main} and Corollary~\ref{cor:main} by induction on $g$. The case $g=0$ is trivial, thus we may assume that $g\geq 1$ and that Theorem~\ref{th:main} holds for all genera up to $(g-1)$. Using the third point of Lemma, Corollary~\ref{cor:main} also holds for all genera up to $g-1$. 
% Then we show the following lemma.
%\begin{lemma}
%If Theorem~\ref{??} holds for all genera up to $(g-1)$, then Corollary~\ref{??} also holds for all genera up to $g-1$. 
%\end{lemma}
%\begin{proof} Let $\beta$ be a class of the form $\zeta_{\Gamma*}(P)$, where  $\Gamma$ is a stable graph of of genus $g'<g$ and with $n$ markings, and $P$ is a monomial in $\psi$ and $\kappa$ classes. We fix some value of $0\leq i\leq g'$. Moreover, we chose any order on the vertices of $\Gamma$, i.e. an identification $V\simeq \{v_1,\ldots,v_k\}$. Then, we have the following identity:
%$$
%P\zeta_{\Gamma}^*\lambda_{g-i}=  \sum_{i_1+\ldots+i_k=i-h^{1}(\Gamma)} P\cdot\bigotimes_{j=1}^k \lambda_{g(v_j)-i_j}.
%$$ 
%Where the $i_j$'s in the sum are non-negative integers. Moreover, for any partition $i_1,\ldots, i_k$ in this sum, as Theorem~\ref{??} holds up to $(g-1)$, we may write each of the $\lambda_{g(v_j)-i_j}$ satisfies ${\rm Loop}_{i_j}$. Thus,  $\beta \cdot  \lambda_{g'-i}$ is a linear combination of decorated graphs with at most $(i-h^{1}(\Gamma))+h^{1}(\Gamma)$ loops. \end{proof}

\begin{lemma}\label{lem:fil2}
If Theorem~\ref{th:main} holds for all genera up to $(g-1)$, then for all vectors of non-negative integers $Z$,    the coefficient of $\xi^i$ in $\alpha(g,Z)$ sits in $R^*_i(\oM_{g,n})$.  
\end{lemma}

\begin{proof}
To prove this lemma, we will show a more general statement: for all pairs $(\underline{g},\underline{Z})$ with $\underline{g}=(g)$ or is of size smaller than $g$, the coefficient of $\xi^i$ in $\alpha(\underline{g},\underline{Z})$ may be expressed as linear combination of classes with at most $i$ loops.

3We prove this more general statement by induction on the size of $\underline{Z}$. It holds trivially if $\underline{Z}$ has size 0. Thus, let $\underline{Z}$ be a vector of size at least one (we assume that $z_{1,1}>0$). Then we recall from Section~\ref{ssec:algo} that $\alpha(\underline{g},\underline{Z})$ is given by $$(\xi+z_1\psi_{1,1})\alpha(\underline{g},\underline{Z}_{1,1})- \sum_{\overline{\Gamma}\in {\rm Bic}(\underline{g},\underline{Z}_{1,1})}m(\overline{\Gamma}) \alpha(\overline{\Gamma}).$$
As the coefficient of $\xi^i$ in $\alpha(\underline{g},\underline{Z}_{1,1})$ may be expressed with disconnected graphs with at most $i$ loops for all $i \geq 0$. This is also the case for $(\xi+z_1\psi_1)\alpha(g,Z_{1,1})$. Besides if $\overline{\Gamma}\in {\rm Bic}(\underline{g},\underline{Z}_{1,1})$, then we will prove the the degree $i$ in $\xi$ of $\alpha(\overline{\Gamma})$ may be expressed with graphs with at most $i$ loops for all $i \geq 0$. Indeed, we recall from Section~\ref{ssec:algo}, that the class $\alpha(\overline{\Gamma})$ is given (up to a coefficient) by:
$$
 \xi^{h^1(\overline{\Gamma})} \zeta_{\Gamma *}\left( f_{-1*}[\PP\oH(\overline{\Gamma})^{[-1]}]   \bigotimes_{v\in \ell^{-1}(i)} \Lambda_v(\xi) \right) \otimes \alpha(\underline{g}^{[0]},\underline{Z}^{[0]}).
$$
By induction assumption, the coefficient of $\xi^i$ in $\alpha(\underline{g}^{[0]},\underline{Z}^{[0]})$ may be expressed with  graphs with at most $i$ loops. Moreover, for each $v$ of level $-1$ the coefficient of $\xi^{i}$ of the contribution of the vertex is a linear combination of classes of the form $\lambda_{g(v)-i} \cdot \beta$ for some tautological class $\beta\in R^*(\oM_{g(v),n(v)})$. Thus using Lemma~\ref{lem:fil1}, it sits in $R^*_i(\oM_{g(v),n(v)})$. Therefore the coefficient of $\xi^i$ of the class 
$$
\left( f_{-1*}[\PP\oH(\overline{\Gamma})^{[-1]}]   \bigotimes_{v\in \ell^{-1}(i)} \Lambda_v(\xi) \right) \otimes \alpha(\underline{g}^{[0]},\underline{Z}^{[0]})
$$
may be expressed by using at most $i$ times the attaching morphism of type loop. QED.
\end{proof}

In order to finish the proof of Theorem~\ref{th:main}, we consider the classes
$$
\alpha(g,z,n)\coloneqq \frac{1}{n!} \cdot \pi_{n*} \alpha\bigg(g,(z,\underset{n\times}{\underbrace{2,\ldots,2}})\bigg),
$$ 
for all non negative integers $g,z,n$. More specifically, we willconsider the class $\alpha(g,1,g-1)$. This class is of degree $g$ and vanishes in $R^*(\PP\oH_{g,1})$ by Proposition~\ref{pr:prop1}. Thus it is of the form:
$$
\alpha(g,1,g-1)= a_g  \Lambda(\xi).
$$
for some rational number $a_g$. Thus, if $a_g\neq 0$, then Lemma~\ref{lem:fil2} implies that the class $\lambda_{g-i}$ sits in $R^*_i(\oM_{g,1})$. Therefore Theorem~\ref{th:main} and Corollary~\ref{cor:main} follow from the following proposition.

\begin{proposition}\label{pr:coeff}
For all $g\geq 1$, we have $a_g=2^{g-1}g$.
\end{proposition}

In order to prove this proposition we will denote $a(g,z,n)\in \QQ$ the coefficient in $\xi^{z+n}$ of $\alpha(g,z,n)$. With this notation, we have $a_g=a(g,1,g-1)$. 

\begin{lemma}\label{lem:identities} The function $a$ satisfies the following identities:
\begin{enumerate}
\item For all $g,z$ we have: $a(g,z,0)=1$.
\item For all $g,z,n$ we have: $a(g,z,n)=a(g,z-1,n)-2 a(g,z+1,n-1)$.
\item For all $g\geq 1$, we have: $a(g,0,g-1)=2^{g-1}(2^g-1)$.
\item If $g\geq 2$, and $0\leq n<g-1$, then we have: 
$$a(g,g-1-n,n)=a(g-1,g-2-n,n)+ 4a(g-1,g-1-n,n-1).$$
\end{enumerate} 
\end{lemma}

\begin{proof}
In order to compute the function $a(g,z,n)$ one only need to compute the restriction of $\alpha(g,z,n)$ to $R^*(\M^{\rm rt}_{g,n},\QQ)$. We use this fact to prove the first two identities. Then, for all $Z=(z_1,\ldots,z_n)$ vectors with $z_1>0$,  the formulas of Section~\ref{ssec:algo} take the following simpler form by Proposition~\ref{pr:prop2}:
$$
\alpha(g,Z)= (\xi+z_1\psi_1) \alpha(g,Z_1) - \sum_{1\in I\subset [\![1,n]\!]}  (z_I+1) \delta_{0,I} \cdot \alpha(g,Z_1) \text{ in $R^*(\M_{g,n}^{\rm rt})[\xi]$}.
$$
where the sum is over all subsets of $I$ of $[\![1,\ldots,n]\!]$, and we recall that $z_I$ stands for $-1+\sum_{i\in I}z_i$. In particular if $n=1$,  $\alpha(g,(z))=\prod_{j=1}^z (\xi+j\psi_1)$ in $R^*(\M_{g,1}^{\rm rt})[\xi]$, thus implying the first identity of the proposition.\bigskip

To prove the second identity, we denote by $\psi_1^*$ the pull-back of $\psi_1$ along the morphism $\pi_{n-1}:\oM_{g,n}\to \oM_{g,1}$, then we recall that $$\psi_1^*=\psi_1+\sum_{1\in I\subset [\![1,n]\!]}  (z_I+1) \delta_{0,I},$$
thus we obtain:
$$
\alpha(g,Z)= (\xi+z_1\psi_1^*) \alpha(g,Z_1) - \sum_{1\in I\subset [\![1,n]\!]}  (z_I-z_1+1) \delta_{0,I} \cdot \alpha(g,Z_1) \text{ in $R^*(\M_{g,n}^{\rm rt})[\xi]$}.
$$
We apply this expression to the vector of length $(n+1)$: $Z=(z,2,\ldots,2)$. Then in this, case if we push-forward the above expression along $\pi_n$, the only terms in the sum that do not vanish are the one indexed by $I$ of size exactly $2$. Then we obtain:
$$
n! \alpha(g,z,n)=\pi_{n*}\alpha(g,Z)= (\xi+z_1\psi_1) n! \alpha(g,z-1,n) - 2 (n-1)! \alpha(g,z+1,n-1) 
$$
in $R^*(\M_{g,n}^{\rm rt})[\xi]$. Taking the top coefficient in $\xi$ in this expression gives the second identity. \bigskip

To prove the third and fourth identities, we remark that under the constraint $n+z=g-1$, the class $\alpha(g,z,n)$ is of degree $g-1$. Thus, if we denote by $f:\PP\oH_{g,1}\to \oM_{g,1}$, then $a(g,z,n)$ is equal to $f_*\alpha(g,z,n)$ in $R^0(\oM_{g,1})\simeq \QQ$.   Using this remark, we may use the de Joncquières formula as in~\cite{Mul} to show that $a(g,0,g-1)=2^{g-1}(2^g-1)$ (this is also the number of odd spin structures on a curve of genus $g$). \bigskip

To prove the fourth identity, we will compute the intersection of $f_*\alpha(g,z,n)$ with the divisor $\delta_{1,\{1\}}$ to compute $a(g,z,n)$ in terms of evaluations of the function $a$ at smaller genera. For all pairs of vectors $(Z,P)$ of length $n$ and $m$ and satisfying $|Z|-|P|=2g-2$, we denote by:
$$\oM(g,Z,P)=f(\PP\oH(g,Z,P))$$ 
where we recall that $f:\PP\oH_{g,n}(P)\to \oM_{g,n+m}$ is the forgetful morphism of the differential (or simply $\oM(g,Z)$ if $P$ is empty). Then, with this notation, if we impose the constraint $z+n=g-1$, then we have 
$$a(g,z,n)=\frac{1}{n!(g-1-n)!} \pi_{(g-1)*} [\oM(g,(z,\underset{n \times}{\underbrace{2,\ldots,2}},\underset{(g-n) \times}{\underbrace{1\ldots,1}})].
$$

We fix some vector $Z$ of length $g$, size $(2g-2)$ and satisfying $z_1>0$. Then for any choice of $I\subset [\![1,n]\!]$, such that $I$ contains $1$, then we we have the following identity
$$
j_{1,I}^*[\oM(g,Z)] =\left\{ 
\begin{array}{l} 0 \text{ if $z_I=1$}\\
\text{$ [\oM(1,(z_{i,i\in I},-z_I)] \otimes [\oM(g-1,(z_{i,i\notin I}, z_I-2)] $}  \text{ otherwise}.\end{array}\right.
$$
Indeed, this is due to the fact that $\oM(g,Z)$ intersects transversally the divisor $\delta_{1,I}$ along the locus $\oM(g,({z_i}_{i\in I}, -z_I)\times\oM(g,({z_i}_{i\notin I}, z_I-2)$ if $z_I>1$ while this intersection is a locus of co-dimension at least $2$ if $z_I=1$. Then, for dimension reasons, the push-forward of $\delta_{1,I}\cdot [\oM(g,Z)]$ along $\pi_{(g-1)}$ vanishes if $I$ is not of size $2$ (this is due to the fact $\oM(1,Z,P)$ is of dimension $1$ if $P$ is non-trivial). Besides we have the identity
$$\pi_{(g-1)}^*\delta_{1,\{1\}}=\sum_{1\in I\subset [\![1,g]\!]} \delta_{1,I}.$$ 
Therefore in the case of $Z=(z,\underset{n \times}{\underbrace{2,\ldots,2}},\underset{(g-n) \times}{\underbrace{1\ldots,1}})$, we obtain the identity:
\begin{eqnarray*}
a(g,z,n)&=& \pi_*[\oM(1,(z,1,-z-1)] \times a(g-1,z-1,n) \\ &&+ \pi_*[\oM(1,(z,2,-z-2)] \times a(g,z,n-1) 
\end{eqnarray*}
where $\pi$ stands for the forgetful morphism of the second marking. The first summand in the RHS corresponds to the choices of $I=\{z,i\}$ with $z_i=1$, while the second corresponds to $z_i=2$. In both cases we have:
$$
\pi_*[\oM(1,(z,z',-z-z')]=(z')^2,
$$
as there are $(z')^2$ points $x$ on a general elliptic curve $E$ with markings $(x_1,x_3)$ that satisfy  the equation $z'\cdot [x]=(z+z')\cdot [x_3]-z\cdot [x_1]$ in the Picard group of $E$. This finishes the proof of the fourth identity. 
\end{proof}

\begin{proof}[End of the proof of Proposition~\ref{pr:coeff}] 
We define two sequences
\begin{equation*}
    u_{g,n}:=a(g, g-n, n) \ \ \text{and} \ \ w_{g,n}=a(g,g-n-1,n)
\end{equation*}
The first part of the above lemma then readily translates to $ u_{g,n}=w_{g,n}-2u_{g,n-1}$
which in turn, implies that 
\begin{equation}\label{eq2}
   u_{g,n}=\sum_{i=0}^n(-2)^iw_{g,n-i},
\end{equation}
since $u_{g,0}=w_{g,0}=1$. Furthermore, from the last part of Lemma~\ref{lem:identities} we have 
$w_{g,n}=w_{g-1,n}+4w_{g-1,n-1}$, if $n<g-1$. If $g>1$, we would like to compute $u_{g,g-1}$, by applying $n=g-1$ in \eqref{eq2}:
\begin{align*}
    u_{g,g-1} &= w_{g,g-1} + \sum_{i=1}^{g-1}(-2)^iw_{g,g-1-i} \\
              &= w_{g,g-1} + \sum_{i=1}^{g-1}(-2)^iw_{g-1,g-1-i} + 4 \sum_{i=1}^{g-2}(-2)^iw_{g-1,g-2-i} \\
              &= w_{g,g-1} - 2\sum_{i=0}^{g-2}(-2)^iw_{g-1,g-2-i} + 4 (u_{g-1,g-2} -w_{g-1,g-2}) \\
              &= w_{g,g-1} -2u_{g-1,g-2} + 4u_{g-1,g-2} - 4w_{g-1,g-2} \\
              &=2u_{g-1,g-2}+2^{g-1}.
\end{align*}
Note that we used the third identity of Lemma~\ref{lem:identities}: $w_{g,g-1}=2^{g-1}(2^g-1)$. Putting everything together, we have $a_{g}=2a_{g-1}+2^{g-1}$ and $a_{1}=1$. Therefore, we have $a_g=2^{g-1}g.$
\end{proof}

%\section{Explicit values for $\lambda_g$}

%\bibliographystyle{halpha}
%\bibliography{biblio}

\end{document}